\documentclass[10pt]{article}

\usepackage{amsmath,amsfonts,amssymb,amsthm}
\usepackage{color}
\usepackage{array}
\usepackage{mathrsfs}
\usepackage{dcolumn}
\usepackage{longtable}
\usepackage{hhline}
\usepackage{bm}

\newcommand{\bracketed}[2]{\genfrac{[}{]}{0pt}{0}{#1}{#2}}

\newtheorem{thm}{Theorem}[section]
\newtheorem{defn}[thm]{Definition}

\newtheorem{rema}[thm]{Remark}

\title{\bf Explicit Formulas for the First Form $(q,r)$-Dowling Numbers and $(q,r)$-Whitney-Lah Numbers}
\author{{\large\bf Roberto B. Corcino$^{1,2}$}\\{\large\bf Jay M. Ontolan$^{1,2}$}\\ {\large\bf Maria Rowena S. Lobrigas$^{2}$}\\
$^{1}$Research Institute for Computational Mathematics and Physics\\
Cebu Normal University, Cebu City, Philippines 6000\\
$^{2}$Department of Mathematics\\ Cebu Normal University, Cebu City, Philippines 6000
}

\date{}

\begin{document}

\maketitle

\begin{abstract}
	In this paper, a $q$-analogue of $r$-Whitney-Lah numbers, also known as $(q,r)$-Whitney-Lah number, denoted by $L_{m,r}[n,k]_q$ is defined using the triangular recurrence relation. Several fundamental properties for the $q$-analogue are established such as vertical and horizontal recurrence relations, horizontal and exponential generating functions. Moreover, an explicit formula for $(q,r)$-Whitney-Lah number is derived using the concept of $q$-difference operator, particularly, the $q$-analogue of Newton's Interpolation Formula (the umbral version of Taylor series). Furthermore, an explicit formula for the first form  $(q,r)$-Dowling numbers is obtained which is expressed in terms of $(q,r)$-Whitney-Lah numbers and $(q,r)$-Whitney numbers of the second kind.

\end{abstract}

\section{Introduction}
The Lah numbers $L(n,k)$ are the connection constants between the rising factorial and falling factorial polynomial bases and count partitions of $n$ distinct objects into $k$ blocks, where objects within a block are ordered (termed Laguerre Configurations) \cite{lindsay}. The classical (signless) Lah numbers $L(n,k)=\frac{n!}{k!}(^{n-1}_{k-1})$ may be expressed in terms of the Stirling numbers $s(n,k)$ and $S(n,k)$ of the first and second kind, respectively:
\begin{equation}
L(n,k)=\sum_{j=k}^{n}s(n,j)S(j,k).
\end{equation}
Cheon and Jung \cite{cheon} defined the $r$-Whitney-Lah numbers $L_{m,r}(n,k)$ in terms of the $r$-Whitney numbers of the first kind $w_{m,r}(n,k)$ and second kind $W_{m,r}(n,k)$
\begin{equation}
L_{m,r}(n,k)=\sum_{j=k}^{n}w_{m,r}(n,j)W_{m,r}(j,k),
\end{equation}
which generalizes the identity for the Lah numbers in terms of the Stirling numbers of the first kind $s(n,k)$ and second kind $S(n,k)$.

A $q$-analogue is a mathematical expression parameterized by a quantity $q$ that generalizes a known expression and reduces to the known expression in the limit, as $q$ approaches $1$.

R.B. Corcino et al\cite{cor4} established a $q$-analogue of $S_{n,k}(a)$ denoted by $ \sigma{[n,k]^{\beta, r}_q}$, defined by
$$\sigma{[n,k]^{\beta, r}_q}= \sigma{[n-1,k-1]^{\beta, r}_q}+([k\beta]_q+[r]_q)\sigma{[n-1,k]^{\beta, r}_q}$$
and obtained some properties including vertical and horizontal recurrence relation, horizontal generating function, explicit formula, as well as exponential generating function, rational generating function, and explicit formula in homogeneous symmetric form. It is well-known that the Rucinski-Voigt numbers are also known as
r-Whitney numbers of the second kind, denoted by $W_{m,r}(n,k)$ by Mezo\cite{mezo}.

Corcino and Montero \cite{cor4} defined a $q$-analogue of $r$-Whitney numbers of the second kind which are exactly the same numbers with the Rucinski-Voigt numbers, in a form  of triangular recurrence relation. 

On the other hand, Cheon and Jung \cite{cheon} defined the $r$-Whitney-Lah numbers, denoted by $L_{m,r}(n,k)$, in terms of the $r$-Whitney numbers of the first kind $w_{m,r}(n,k)$ and the second kind $W_{m,r}(n,k)$
\begin{equation*}
L_{m,r}(n,k)=\sum^n_{j=k}w_{m,r}(n,j)W_{m,r}(j,k),
\end{equation*}
which generalizes the identity for the Lah numbers in terms of the Stirling numbers of the first kind and the second kind
\begin{equation*}
(-1)^nL(n,k)=\sum^n_{j=k}(-1)^js(n,j)S(j,k).
\end{equation*}

Other properties for $L_{m,r}(n,k)$ were established which include its horizontal generating function
\begin{equation*}
\langle x+2r|m \rangle _n=\sum^n_{k=0}L_{m,r}(n,k)(x|m)_k,
\end{equation*}
where 

$$\langle x+2r|m \rangle _n=\begin{cases}
(x+2r)\ldots (x+2r+(n-1)m), \ \ \ n\geq 1\\
0, \ \ \ \ \ \ \ \ \ \ \ \ \ \ \ \ \ \ \ \ \ \ \ \ \ \ \ \ \ \ \ \ \ \ \ \ \ \ \ \ \ \ \ n=0.
\end{cases}$$

and the triangular recurrence relation
\begin{equation*}
L_{m,r}(n,k)=L_{m,r}(n-1,k-1)+(2r+(n+k-1)m)L_{m,r}(n-1,k),
\end{equation*}
with $L_{m,r}(n,n)=1$ for $n\geq0$ and $L_{m,r}(n,n)=0$ for $n<k$, or $n,k <0$. We can use the triangular recurrence relation to generate the first values of $L_{m,r}(n,k)$.

Cillar and Corcino \cite{cillar} defined the $q$-Analogue of $r$-Whitney-Lah numbers which is parallel to Cheon and Jung's definition of $r$-Whitney-Lah numbers
\begin{equation*}
L_{\beta,r}[n,k]_q=\sum^n_{j=k}\phi_{\beta,r}[n,j]_q\sigma[j,k]^{\beta,r}_q.
\end{equation*}
where $\phi_{\beta,r}[n,j]_q$ and $\sigma[j,k]^{\beta,r}_q$ is equivalent to $w_{\beta,r}[n,j]_q$ and $W_{\beta,r}[n,j]_q$, respectively.

They also obtained some combinatorial properties. The following are the results:
\begin{enumerate}
	\item \textit{Horizontal generating function}
	\begin{equation*}
	\langle t+2[r]_q|[\beta]_q\rangle_n=\sum^n_{k=0}L_{\beta,r}[n,k]_q(t|[\beta]_q)_k
	\end{equation*}
	
	where

	$$\langle t+2[r]_q|[\beta]_q\rangle_n=\begin{cases} \Pi ^{n-1}_{i=0}(t+2[r]_q+[i\beta]_q), \ \ \ n>0,\\
		1, \ \ \ \ \ \ \ \ \ \ \ \ \ \ \ \ \ \ \ \ \ \ \ \ \ \ \ \ \ \ n=0.
\end{cases}$$	
			
	\item \textit{Recurrence relation}
	\begin{equation*}
	L_{\beta,r}[n,k]_q=L_{\beta,r}[n-1,k-1]_q+(2[r]_q+[k\beta]_q+[(n-1)\beta]_q)L_{\beta,r}[n-1,k]_q,
	\end{equation*}
	with $L_{\beta,r}[0,0]_q=1$ and $L_{\beta,r}[n,k]_q=0$ for $n<k$ or $n,k<0$.
\end{enumerate}

Recently, a $q$-analogue of $r$-Whitney numbers of the second kind $W_{m,r}[n,k]_q$, also known as $(q,r)$-Whitney number of the second kind, was introduced in \cite{latayada, jen} by means of the following triangular recurrence relation
\begin{equation}\label{qrDN}
W_{m,r}[n,k]_{q}= q^{m(k-1)-r}W_{m,r}[n-1,k-1]_q+[mk-r]_{q} W_{m,r}[n-1,k]_{q}.
\end{equation}
From this definition, two more forms of the $q$-analogue were defined in \cite{jen, latayada} as
\begin{align}
W^*_{m,r}[n,k]_q&:=q^{-kr-m\binom{k}{2}}W_{m,r}[n,k]_q\label{2ndform}\\
\widetilde{W}_{m,r}[n,k]_q&:=q^{kr}W^*_{m,r}[n,k]_q=q^{-m\binom{k}{2}}W_{m,r}[n,k]_q,\label{3rdform}
\end{align}
where $W^*_{m,r}[n,k]_q$ and $\widetilde{W}_{m,r}[n,k]_q$ denote the second and third forms of the $q$-analogue, respectively. Corresponding to the $q$-analogues in equations \eqref{qrDN}, \eqref{2ndform} and \eqref{3rdform}, three forms of $q$-analogues for $r$-Dowling numbers (also known as $(q,r)$-Dowling numbers) were also defined as follows:
\begin{align}
{D}_{m,r}[n]_q&:=\sum_{k=0}^nW_{m,r}[n,k]_q\label{1stformqDow}\\
{D}^*_{m,r}[n]_q&:=\sum_{k=0}^nW^*_{m,r}[n,k]_q\label{2ndformqDow}\\
\widetilde{D}_{m,r}[n]_q&:=\sum_{k=0}^n\widetilde{W}_{m,r}[n,k]_q.\label{3rdformqDow}
\end{align}
The focus in this paper is on the first form of the $q$-analogue. The following are the results obtained for the first form $(q,r)$-Whitney numbers of the second kind:
\begin{enumerate}
    	\item \textit{Vertical and Horizontal Recurrence Relations}
	\begin{equation}
	W_{m,r}[n+1,k+1]_q=q^{mk+r}\sum_{j=k}^{n}[m(k+1)+r]^{n-j}_qW_{m,r}[j,k]_q
	\end{equation}
	\begin{equation}
	W_{m,r}[n,k]_q=\sum_{j=0}^{n-k}(-1)^jq^{-r-m(k+j)}\frac{r_{k+j+1,q}}{r_{k+1,q}}W_{m,r}[n+1,k+j+1]_q
	\end{equation}
	\noindent respectively, where
	\begin{equation*}
	r_{i,q}=\prod_{h=1}^{i-1}q^{-r-mh+m}[mh+r]_q
	\end{equation*}
	\noindent with initial value $W_{m,r}[0,0]_q=1.$
	
	\item \textit{Horizontal Generating Function}
	\begin{equation}\label{defn_qWNSK}
	\sum_{k=0}^{n}W_{m,r}[n,k]_q[t-r|m]_{k,q}=[t]^n_q.
	\end{equation}
	
	\item \textit{Explicit Formula} 
	\begin{equation}
	W_{m,r}[n,k]_q=\frac{1}{[k]_{q^m}![m]^k_q}\sum_{j=0}^{k}(-1)^{k-j}q^{m\binom{k-j}{2}}\bracketed{k}{j}_{q^m}[jm+r]^n_q. 
	\end{equation}
	
	\item \textit{Exponential Generating Function}
	\begin{equation}\label{exp_qrWhit}
	\sum_{n\geq 0}W_{m,r}[n,k]_q\frac{[t]^n_q}{[n]_q!}=\frac{1}{[k]_{q^m}![m]^k_q}[\Delta_{q^m,m^{k}}e_q([x+jm+r]_q[t]_q)]_{x=0}.
	\end{equation}
	
	\item \textit{Rational Generating Function}
	\begin{equation}\label{rationalGF}
	\Psi_k(t)=\sum_{n\geq k}W_{m,r}[n,k]_q[t]^n_q=\frac{q^{m\binom{k}{2}+kr}[t]^k_q}{\prod_{j=0}^{k}(1-[mj+r]_q[t]_q)}.
	\end{equation}
\end{enumerate}
The purpose of introducing these new $q$-analogues of $r$-Whitney and $r$-Dowling numbers is to address the conjecture in the paper of Corcino-Corcino \cite{Cor-Cor1}. That is, to establish the Hankel transform of the $q$-analogue of $(r,\beta)$-Stirling numbers, which are exactly the $r$-Dowling numbers. In establishing the Hankel transforms of these three forms of $(q,r)$-Dowling numbers, the Hankel transform of the third form was the first one being established and followed by the second form. For the first form, it is still on the process of constructing the method that can possibly be used to derive it.

\smallskip
In this paper,  $(q,r)$-Whitney-Lah numbers will be introduced and properties of these  number will be establshed using the approach employed in \cite{jen}. The method used in the paper of Cillar and Corcino \cite{cillar} will also be used to obtain the properties of the $(q,r)$-Whitney-Lah numbers. Moreover, this paper is concluded by deriving an explicit formula for the first form $(q,r)$-Dowling numbers expressed in terms of $(q,r)$-Whitney-Lah numbers and $(q,r)$-Whitney numbers of the second kind.

\bigskip
\section{$(q,r)$-Whitney-Lah Numbers and Their Recurrence Relations}
The definition of the $q$-analogue of $r$-Whitney-Lah numbers, also known as the $(q,r)$-Whitney-Lah numbers, is given as follows.

\smallskip
\begin{defn}
\normalfont The $(q,r)$-Whitney-Lah numbers $L_{m,r}[n,k]_q$ are defined by
\begin{equation}
L_{m,r}[n,k]_q=q^{2r+m(k-1)+m(n-1)} L_{m,r}[n-1,k-1]_q+[2r+km+(n-1)m]_q L_{m,r}[n-1,k]_q,\label{3.1}
\end{equation}
with
\[ L_{m,r}[n,k]_q= \left\{ 		\begin{array}{rll}
		0 \mbox & {if} & n<k \ \ or \ \ n, k <0,\\
		1 \mbox & {if} & n=k \ \ and \ \ n\geq 0.
		\end{array}\right.
	\]
and
$$[t-k]_q=\frac{1}{q^k}([t]_q-[k]_q).$$
\end{defn}

\smallskip
\begin{rema}\rm
When $q\rightarrow1$, we obtained the triangular recurrence relation of $r$-Whitney-Lah numbers, $L_{m,r}(n,k)$, defined by Cheon and Jung \cite{cheon}:
$$L_{m,r}(n,k)=L_{m,r}(n-1,k-1)+(2r+km+(n-1)m)L_{m,r}(n-1,k)$$
with $L_{m,r}(n,k)=1$ for $n\geq 0$ and $n=k$,\\
and $L_{m,r}(n,k)=0$ for $n<k$ or $n,k <0$.
\end{rema}

\smallskip
\begin{rema}\rm 
It can easily be verified that
$$L_{m,r}[n,0]=[2r+(n-1)m]^n_q$$
\end{rema}

By applying (\ref{3.1}), we can easily obtain two other forms of recurrence relations and generating function.

\begin{thm}
\normalfont A $q$-analogue of $r$-Whitney-Lah numbers, $L_{m,r}[n,k]_q$, satisfies the following vertical recurrence relations,
$$L_{m,r}[n+1,k+1]_q=\sum_{j=k}^nq^{[2r+mk+m(n-j)]}\left(\prod_{i=0}^k [2r+(k+1)m+(n-i)m]_q\right) L_{m,r}[j,k]_q$$
with initial value $L_{m,r}[0,0]_q=1$.
\end{thm}

\begin{proof}
We replace $n$ by $n+1$ and $k$ by $k+1$, then using (\ref{3.1}), we have
\begin{equation*}
L_{m,r}[n+1,k+1]_q=q^{2r+mk+mn}L_{m,r}[n,k]_q+[2r+(k+1)m+(n)m]_qL_{m,r}[n,k+1]_q
\end{equation*}

Applying (\ref{3.1}) repeatedly we have,
\begin{align*}
L_{m,r}[n+1,k+1]_q&=q^{2r+mk+mn}L_{m,r}[n,k]_q+[2r+(k+1)m+nm]_q \\
&\;\;\;\;\;\;\lbrace q^{2r+mk+m(n-1)}L_{m,r}[n-1,k]_q+[2r+(k+1)m+(n-1)m]_qL_{m,r}[n-1,k+1]_q\rbrace\\
&=q^{2r+mk+mn}L_{m,r}[n,k]_q+q^{2r+mk+m(n-1)}[2r+(k+1)m+nm]_qL_{m,r}[n-1,k]_q\\
&\;\;\;\;\;\;+q^{2r+mk+m(n-2)}[2r+(k+1)m+nm]_q[2r+(k+1)m+(n-1)m]_qL_{m,r}[n-2,k]_q\\
&\;\;\;\;\;\;+\ldots+ \ q^{2r+mk+m(k+1)} \lbrace [2r+(k+1)m+nm]_q[2r+(k+1)m+(n-1)m]_q\\
&\;\;\;\;\;\;\ldots [2r+(k+1)m+(n-k)m]_q \rbrace L_{m,r}[k+1,k+1]_q
\end{align*}
Using the fact that $L_{m,r}[k+1,k+1]_q=L_{m,r}[k,k]_q$ we have,
$$L_{m,r}[n+1,k+1]_q=\sum_{j=k}^nq^{[2r+mk+m(n-j)]}\left(\prod_{i=0}^k[2r+(k+1)m+(n-i)m]_q\right)L_{m,r}[j,k]_q$$
\end{proof}

\begin{thm}
\normalfont A $q$-analogue $L_{m,r}[n,k]_q$ satisfies the horizontal recurrence relation
$$L_{m,r}[n,k]_q=\sum_{j=0}^{n-k}(-1)^j\ q^{-2r-m(k+j)-nm}\ \frac{r_{k+j+1,q}}{r_{k+1,q}}\ L_{m,r}[n+1,k+j+1]_q$$
where
$$r_{i,q}=\prod_{h=1}^{i-1}q^{-2r-mh-nm+m}[mh+2r+nm]_q$$
with initial value $L_{m,r}[0,0]_q=1$.
\end{thm}
\begin{proof}
Applying (\ref{3.1}) we have,
\begin{align*}
RHS&=\sum_{j=0}^{n-k}(-1)^jq^{-2r-m(k+j)-nm} \frac{r_{k+j+1,q}}{r_{k+1,q}}\\
&\;\;\;\{q^{2r+m(k+j+1-1)+m(n+1-1)}L_{m,r}[n,k+j]_q+[2r+(k+j+1)m+nm]_qL_{m,r}[n,k+j+1]_q\}\\
&=\sum_{j=0}^{n-k}(-1)^j \frac{r_{k+j+1,q}}{r_{k+1,q}} L_{m,r}[n,k+j]_q\\
&\;\;\;\;\;\;+\sum_{j=1}^{n-k+1}(-1)^{j-1}q^{-2r-m(k+j-1)-nm} \frac{r_{k+j,q}}{r_{k+1,q}} [2r+(k+j)m+nm]_qL_{m,r}[n,k+j]_q\\
&=\sum_{j=0}^{n-k}(-1)^j \frac{r_{k+j+1,q}}{r_{k+1,q}} L_{m,r}[n,k+j]_q+\sum_{j=1}^{n-k}(-1)^{j-1}q^{-2r-m(k+j-1)-nm}\\
&\;\;\;\;\;\frac{\prod_{h=1}^{k+j-1}q^{-2r-mh-nm+m}[mh+2r+nm]_q}{\prod_{h=1}^{k}q^{-2r-mh-nm+m}[mh+2r+nm]_q} [2r+(k+j)m+nm]_qL_{m,r}[n,k+j]_q
\end{align*}
Note that
\begin{align*}
&\prod_{h=1}^{k+j-1}q^{-2r-mh-nm+m}[mh+2r+nm]_q=(q^{-2r-m(1)-nm+m}[m(1)+2r+nm]_q)\\
&\;\;\;\;\;\;\;\;\;\;\;\;\;\;\;(q^{-2r-m(2)-nm+m}[m(2)+2r+nm]_q)(q^{-2r-m(3)-nm+m}[m(3)+2r+nm]_q)\\
&\;\;\;\;\;\;\;\;\;\;\;\;\;\;\;\;\;\;\;\;\ldots (q^{-2r-m(k+j-1)-nm+m}[m(k+j-1)+2r+nm]_q)
\end{align*}
Now, simplifying the RHS we have,
\begin{align*}
RHS&=\sum_{j=0}^{n-k}(-1)^j\ \frac{r_{k+j+1,q}}{r_{k+1,q}}\ L_{m,r}[n,k+j]_q+\frac{1}{{\prod_{h=1}^{k}q^{-2r-mh-nm+m}[mh+2r+nm]_q}}\\
&=\sum_{j=0}^{n-k}(-1)^j\ \frac{r_{k+j+1,q}}{r_{k+1,q}}\ L_{m,r}[n,k+j]_q+\frac{1}{{\prod_{h=1}^{k}q^{-2r-mh-nm+m}[mh+2r+nm]_q}}\\
&\;\;\;\;\;\;\sum_{j=1}^{n-k}(-1)^{j-1}\{(q^{-2r-m(1)-nm+m}[m(1)+2r+nm]_q)(q^{-2r-m(2)-nm+m}[m(2)+2r+nm]_q)\\
&\;\;\;\;\;\;(q^{-2r-m(3)-nm+m}[m(3)+2r+nm]_q)\ldots (q^{-2r-m(k+j-1)-nm+m}[m(k+j-1)+2r+nm]_q)\}\\
&\;\;\;\;\;\;q^{-2r-m(k+j)-nm+m}[2r+(k+j)m+nm]_qL_{m,r}[n,k+j]_q\\
&=\sum_{j=0}^{n-k}(-1)^j\ \frac{r_{k+j+1,q}}{r_{k+1,q}}\ L_{m,r}[n,k+j]_q+\sum_{j=1}^{n-k}(-1)^{j-1} \frac{r_{k+j+1,q}}{r_{k+1,q}} \ L_{m,r}[n,k+j]_q\\
&=\frac{r_{k+1,q}}{r_{k+1,q}} \ L_{m,r}[n,k]_q=L_{m,r}[n,k]_q
\end{align*}
\end{proof}

\section{Explicit Formula and Generating Functions}	
This property is necessary to obtain the exponential generating function and explicit formula of $L_{m,r}[n,k]_q$. The following theorem contains the horizontal generating function for $L_{m,r}[n,k]_q$.

\begin{thm}
A horizontal generating function of $L_{m,r}[n,k]_q$ is given by
\begin{equation}
\sum_{k=0}^n \ L_{m,r}[n,k]_q \ [t \vert m]_{k,q} = [t+2r \vert m]_{\overline{n},q} \label{4.1}
\end{equation}
where
\[ [t+2r \vert m]_{\overline{n},q}= \left\{ 		\begin{array}{rll}
		\prod^{n-1}_{i=0}[t+2r+im]_q, \mbox & {if} & n\geq1,\\
		1, \ \ \ \ \ \ \ \ \ \ \ \ \ \ \ \ \ \ \ \mbox & {if} & n=0.
		\end{array}\right.
	\]
\end{thm}

\begin{proof}
To prove this, we will use induction on n.
We now verify that (\ref{4.1}) holds when n=0.
$$L_{m,r}[0,0]_q \ [t \vert m]_{0,q}=1 \cdot 1=1=[t+2r \vert m]_{\overline{0},q}$$.
Suppose that it is true for some $n>0$. Then,
$$\sum^n_{k=0}L_{m,r}[n,k]_q \ [t \vert m]_{k,q}=[t+2r \vert m]_{\overline{n},q}$$.
Next, we show that
$$\sum^{n+1}_{k=0}L_{m,r}[n+1,k]_q \ [t \vert m]_{k,q}=[t+2r \vert m]_{\overline{n+1},q}$$
Using (\ref{3.1})  and the fact that. $[t \vert m]_{k+1,q}=[t \vert m]_{k,q}[t-km]_q$, we get
\begin{equation*}
\sum^{n+1}_{k=0}L_{m,r}[n+1,k]_q \ [t \vert m]_{k,q}=\sum^{n+1}_{k=0} \{ q^{2r+m(k-1)+mn}L_{m,r}[n,k-1]_q+[2r+km+nm]_q L_{m,r}[n,k]_q \}[t \vert m]_{k,q}
\end{equation*}
\begin{align*}
&\;\;\;\;\;\;\;\;\;\;\;\;=\sum^{n+1}_{k=0} q^{2r+m(k-1)+mn}L_{m,r}[n,k-1]_q [t \vert m]_{k,q}+ \sum^{n+1}_{k=0} [2r+km+nm]_q L_{m,r}[n,k]_q [t \vert m]_{k,q}\\
&\;\;\;\;\;\;\;\;\;\;\;\;=\sum^{n}_{k=0} q^{2r+mk+mn}L_{m,r}[n,k]_q [t \vert m]_{k+1,q}+ \sum^{n}_{k=0} [2r+km+nm]_q L_{m,r}[n,k]_q [t \vert m]_{k,q}\\
&\;\;\;\;\;\;\;\;\;\;\;\;=\sum^{n}_{k=0} q^{2r+mk+mn}L_{m,r}[n,k]_q [t \vert m]_{k,q}[t-km]_q+ \sum^{n}_{k=0} [2r+km+nm]_q L_{m,r}[n,k]_q [t \vert m]_{k,q}
\end{align*}
Note that $[t-k]_q= \frac{1}{q^k}([t]_q-[k]_q)$. Since $[t-km]_q=[t+2r+nm-2r-nm-km]_q=[(t+2r+nm)-(2r+km+nm)]_q$, then
$$[t-km]_q=\frac{1}{q^{2r+mk+mn}}([t+2r+nm]_q-[2r+km+nm]_q)$$.
Thus, 
\begin{align*}
\sum^{n+1}_{k=0}L_{m,r}[n+1,k]_q \ [t \vert m]_{k,q}&=\sum^{n}_{k=0} q^{2r+mk+mn}L_{m,r}[n,k]_q [t \vert m]_{k,q}
\end{align*}
$$\{q^{-(2r+mk+mn)}([t+2r+nm]_q-[2r+km+nm]_q)\}+ \sum^{n}_{k=0} [2r+km+nm]_q L_{m,r}[n,k]_q [t \vert m]_{k,q}$$
\begin{align*}
&=\sum^{n}_{k=0} L_{m,r}[n,k]_q [t \vert m]_{k,q}([t+2r+nm]_q-[2r+km+nm]_q)
\end{align*}
$$+ \sum^{n}_{k=0} [2r+km+nm]_q L_{m,r}[n,k]_q [t \vert m]_{k,q}$$
\begin{align*}
\ \ \ \ \ \ \ \ \ \ \ \ &=\sum^{n}_{k=0} L_{m,r}[n,k]_q [t \vert m]_{k,q} \{[t+2r+nm]_q-[2r+km+nm]_q + [2r+km+nm]_q \}
\end{align*}
\begin{align*}
&=[t+2r \vert m]^n_q [t+2r+nm]_q \ \ \ \ \ \ \ \ \ \ \ \ \ \ \ \ \ \ \ \ \ \ \ \ \ \ \ \ \ \ \ \ \ \ \ \ \ \ \ \ \ \\
&=[t+2r \vert m]_{\overline{n+1},q} \ \ \ \ \ \ \ \ \ \ \ \ \ \ \ \ \ \ \ \ \ \ \ \ \ \ \ \ \ \ \ \ \ \ \ \ \ \ \ \ \ 
\end{align*}
This proves the theorem.
\end{proof}

Next, we established the explicit formula for $L_{m,r}[n,k]_q$ using the horizontal generating function of $L_{m,r}[n,k]_q$. Also, the exponential generating function for $L_{m,r}[n,k]_q$ is obtained using the explicit formula.

The explicit formula for the $q$-difference operator is as follows
\begin{equation}
\bigtriangleup^h_{q,k}f(x)=\sum^k_{j=0}(-1)^{k-j}q^{(^{k-j}_2)}[^k_j]f(x+jh). \label{4.2.1}
\end{equation}
The new $q$-analogue of Newton's Interpolation Formula in \cite{kim} is given by
$$f_q(x)=a_0+a_1[x-x_o]_q+\ldots+a_m[x-x_o]_q[x-x_1]_q\ldots[x-x_{m-1}]_q,$$
which is equivalent to
\begin{equation*}
f_q(x)=f_q(x_0)+\frac{\triangle_{q^h,h}f_q(x_0)[x-x_0]_q}{[1]_{q^h}![h]_q}+\frac{\triangle_{q^h2,h}f_q(x_0)[x-x_0]_q[x-x_1]_q}{[2]_{q^h}![h]^2_q}
\end{equation*}
$$+\ldots+\frac{\triangle_{q^hm,h}f_q(x_0)[x-x_0]_q[x-x_1]_q\ldots[x-x_{m-1}]_q}{[m]_{q^h}![h]^m_q}$$
where $x_k=x_0+kh, k=1,2,\ldots$ such that if $x_0=0$ and $h=m$, we have
\begin{equation*}
f_q(x)=f_q(x_0)+\frac{\triangle_{q^m,m}f_q(0)[x]_q}{[1]_{q^m}![m]_q}+\frac{\triangle_{q^m2,m}f_q(0)[x]_q[x-m]_q}{[2]_{q^m}![m]^2_q}
\end{equation*}
$$+\ldots+\frac{\triangle_{q^mm,m}f_q(0)[x]_q[x-m]_q\ldots[x-m(m-1)]_q}{[m]_{q^m}![m]^m_q}$$
By (\ref{4.1}) with $t=x$, we get
$$\sum^n_{k=0}L_{m,r}[n,k]_q[x \vert m]_{k,q}=[x+2r \vert m]_{\overline{n},q}$$
which can be rewritten as
$$\sum^n_{k=0}L_{m,r}[n,k]_q[x]_q[x-m]_q[x-2m]_q \ldots [x-(k-1)m]_q=[x+2r \vert m]_{\overline{n},q}.$$
Suppose $f_q(x)=[x+2r \vert m]_{\overline{n},q}$ and
$$L_{m,r}[n,k]_q=\frac{\triangle^k_{q^m,m}f_q(0)}{[k]_{q^m}![m]^k_q}=\frac{1}{[k]_{q^m}![m]^k_q} \cdot \bigtriangleup^k_{q^m,m}f_q(0)$$
Note that applying the above Newton's Interpolation Formula and the identity in \eqref{4.2.1}, we have
\begin{align*}
\bigtriangleup^f_{q^m,m}f_q(x)&=\sum^k_{j=0}(-1)^{k-j}q^{m(^{k-j}_2)}[^k_j]_{q^m}f_q(x+jm)\\
&=\sum^k_{j=0}(-1)^{k-j}q^{m(^{k-j}_2)}[^k_j]_{q^m}[x+jm+2r \vert m]_{\overline{n},q}.
\end{align*}
Evaluate at $x=0$,
$$\triangle^f_{q^m,m}f_q(0)=\sum^k_{j=0}(-1)^{k-j}q^{m(^{k-j}_2)}[^k_j]_{q^m}[2r+jm \vert m]_{\overline{n},q}.$$
Thus,
$$L_{m,r}[n,k]_q=\frac{1}{[k]_{q^m}![m]^k_q}\sum^k_{j=0}(-1)^{k-j}q^{m(^{k-j}_2)}[^k_j]_{q^m}[2r+jm \vert m]_{\overline{n},q}.$$
So, the explicit formula for $L_{m,r}[n,k]_q$ is as follows:

\begin{thm}
The explicit formula for $L_{m,r}[n,k]_q$ is given by
\begin{equation}
L_{m,r}[n,k]_q=\frac{1}{[k]_{q^m}![m]^k_q}\sum^k_{j=0}(-1)^{k-j}q^{m(^{k-j}_2)}[^k_j]_{q^m}[2r+jm \vert m]_{\overline{n},q} \label{4.2.2}
\end{equation}
\end{thm}

\begin{rema}\rm
When $q\rightarrow 1$, the above theorem reduces to, 
$$L_{m,r}(n,k)=\frac{1}{k!m^k}\sum^k_{j=0}(-1)^{k-j}(^k_j)(2r+jm \vert m)_{\overline{n}},$$
which is the explicit formula of $L_{m,r}(n,k)$, where
$$(t \vert m)_{\overline{n}}=t(t+m)\ldots (t+(n-1)m).$$
\end{rema}

\begin{rema}\rm
For brevity, \eqref{4.2.2} can be written as
$$L_{m,r}[n,k]_q=\frac{1}{[k]_{q^m}![m]^k_q}[\triangle^k_{q^m,m}[x+2r \vert m]_{\overline{n},q}]_{x=0}$$
\end{rema}

\begin{thm}
For nonnegative integers $n$ and $k$, and real number $a$, the $q$-analogue $L_{m,r}[n,k]_q$ has an exponential generating function.
\begin{equation}
\sum_{n\geq o}L_{m,r}[n,k]_q \frac{[t]^n_q}{[n]_q!}=\frac{1}{[k]_{q^m}![m]^k_q}[\triangle^k_{q^m,m}(F[x+2r+jm,m,t])_{x=0}, \label{4.13}
\end{equation}
where
$$F[x,m,t]=\sum_{n\geq0}[x \vert m]_{\overline{n},q}\frac{[t]_q^n}{[n]_q!}.$$
\end{thm}

\begin{proof}
Using the explicit formula in \eqref{4.2.2}, we obtained
\begin{align*}
\sum_{n\geq 0}L_{m,r}[n,k] \frac{[t]^n_q}{[n]_q!}&=\sum_{n\geq0}\frac{1}{[k]_{q^m}![m]^k_q}\sum^k_{j=0}(-1)^{k-j}q^{m(^{k-j}_2)}[^k_j]_{q^m}[2r+jm \vert m]_{\overline{n},q} \frac{[t]^n_q}{[n]_q!}\\
&=\sum^k_{j=0}\frac{1}{[k]_{q^m}![m]^k_q}(-1)^{k-j}q^{m(^{k-j}_2)}[^k_j]_{q^m}\sum_{n\geq0}[2r+jm \vert m]_{\overline{n},q}\frac{[t]_q^n}{[n]_q!}.
\end{align*}
Then, we have
\begin{align*}
&=\frac{1}{[k]_{q^m}![m]^k_q}\sum^k_{j=0}(-1)^{k-j}q^{m(^{k-j}_2)}[^k_j]_{q^m}F[2r+jm,m,t]\\
&=\frac{1}{[k]_{q^m}![m]^k_q}[\triangle^k_{q^m,m}(F[x+2r+jm,m,t])_{x=0}
\end{align*}
\end{proof}

\begin{rema}\rm
When $q\rightarrow 1$, \eqref{4.13} becomes
\begin{align*}
\sum_{n\geq 0}L_{m,r}(n,k)\frac{t^n}{n!}&=\frac{1}{k!m^k}\sum^k_{j=0}(-1)^{k-j}[^k_j]F(2r+jm,m,t)
\end{align*}
where
$$F(2r+jm,m,t)=\sum_{n\geq0}(2r+jm \vert m)_{\overline{n},q}\frac{[t]_q^n}{[n]_q!}.$$
\end{rema}

\section{An Explicit Formula for $(q,r)$-Dowling Numbers}
One of the common properties of Lah-type numbers is their relation with both kinds of Stirling-type or Whitney-type numbers. Analogous to this, it is also interesting to express $(q,r)$-Whitney-Lah numbers in terms of $(q,r)$-Whitney numbers. To do this, we need to define the $(q,r)$-Whitney numbers of the first kind.
\begin{defn}\rm
The $(q,r)$-Whitney numbers of the first kind $w_{m,r}[n,k]_q$ are defined as coefficients of the following generating function
\begin{equation}\label{defn_qWNFK}
[t-r|m]_{n,q}=\sum_{k=0}^nw_{m,r}[n,k]_q[t]_q^k.
\end{equation}
\end{defn}
Note that 
\begin{align*}
\sum_{k=0}^{n+1}w_{m,r}[n+1,k]_q[t]_q^k&=[t-r|m]_{n+1,q}=[t-r|m]_{n,q}[t-r-nm]_{q}\\
&=\frac{1}{q^{r+nm}}\left([t]_q-[r+nm]_q\right)\sum_{k=0}^{n}w_{m,r}[n,k]_q[t]_q^k\\
&=\frac{1}{q^{r+nm}}\left(\sum_{k=0}^{n}w_{m,r}[n,k]_q[t]_q^{k+1}-\sum_{k=0}^{n}[r+nm]_qw_{m,r}[n,k]_q[t]_q^k\right)\\
&=\frac{1}{q^{r+nm}}\left(\sum_{k=0}^{n+1}w_{m,r}[n,k-1]_q[t]_q^{k}-\sum_{k=0}^{n+1}[r+nm]_qw_{m,r}[n,k]_q[t]_q^k\right)\\
&=\sum_{k=0}^{n+1}\frac{1}{q^{r+nm}}\left(w_{m,r}[n,k-1]_q-[r+nm]_qw_{m,r}[n,k]_q\right)[t]_q^k.
\end{align*}
Comparing the coefficients of $[t]_q^{k}$ completes the proof the following theorem.
\begin{thm}
The $(q,r)$-Whitney numbers of the first kind $w_{m,r}[n,k]_q$ satisfy the following triangular recurrence relation:
\begin{equation}
w_{m,r}[n+1,k]_q[t]_q^k=\frac{1}{q^{r+nm}}\left(w_{m,r}[n,k-1]_q-[r+nm]_qw_{m,r}[n,k]_q\right)
\end{equation}
with initial condition $w_{m,r}[0,0]_q=1$, $w_{m,r}[n,0]_q=(-1)^{n}q^{-r-(n-1)m}[r+(n-1)m]_q$ and $w_{m,r}[n,k]_q=0$ if $n<k$.
\end{thm}
\begin{table*}[htbp]
	\centering
		\begin{tabular}{|c|c|c|c|c|}
		\hline
		$n/k$ & 0	& 1  & 2 \\
				\hline
	$0$	& 1 & $0$ & $0$\\
		\hline
	$1$	& $-q^{-r}[r]_q$ & $q^{-r}$ & $0$\\
		\hline
	$2$	& $q^{-(r+m)}[r+m]_q$ & $-q^{-(2r+m)}([r]_q+[r+m]_q)$ & $q^{-(2r+m)}$\\
	  \hline
		\end{tabular}
\end{table*}

The next theorem contains the orthogonality relation of $(q,r)$-Whitney numbers of the first and second kinds.
\begin{thm}
The following orthogonality relation holds
\begin{equation}\label{ortho1}
\sum_{k=j}^nw_{m,r}[n,k]_qW_{m,r}[k,j]_q=\sum_{k=j}^nW_{m,r}[n,k]_qw_{m,r}[k,j]_q=\delta_{nj}=
\begin{cases}
0, & if\;\;\; j\neq n\\
1, & if\;\;\; j=n
\end{cases}
\end{equation}
\end{thm}
\begin{proof} Applying equations  \eqref{defn_qWNSK} and \eqref{defn_qWNFK} yields
\begin{align*}
[t-r|m]_{n,q}&=\sum_{k=0}^nw_{m,r}[n,k]_q[t]_q^k=\sum_{k=0}^nw_{m,r}[n,k]_q\sum_{j=0}^kW_{m,r}[k,j]_q[t-r|m]_{j,q}\\
&=\sum_{j=0}^n\sum_{k=j}^nw_{m,r}[n,k]_qW_{m,r}[k,j]_q[t-r|m]_{j,q}=\sum_{j=0}^n\left(\sum_{k=j}^nw_{m,r}[n,k]_qW_{m,r}[k,j]_q\right)[t-r|m]_{j,q}.
\end{align*}
By comparing coefficients, we get
\begin{equation}\label{ortho1}
\sum_{k=j}^nw_{m,r}[n,k]_qW_{m,r}[k,j]_q=
\begin{cases}
0, & if\;\;\; j\neq n\\
1, & if\;\;\; j=n.
\end{cases}
\end{equation}
Similarly, we have
\begin{align*}
[t]_q^n&=\sum_{k=0}^nW_{m,r}[n,k]_q[t-r|m]_{k,q}=\sum_{k=0}^nW_{m,r}[n,k]_q\sum_{j=0}^kw_{m,r}[k,j]_qj[t]_q^n\\
&=\sum_{j=0}^n\sum_{k=j}^nW_{m,r}[n,k]_qw_{m,r}[k,j]_q[t]_q^j=\sum_{j=0}^n\left(\sum_{k=j}^nW_{m,r}[n,k]_qw_{m,r}[k,j]_q\right)[t]_q^j.
\end{align*}
Hence,
\begin{equation}\label{ortho2}
\sum_{k=j}^nW_{m,r}[n,k]_qw_{m,r}[k,j]_q=
\begin{cases}
0, & if\;\;\; j\neq n\\
1, & if\;\;\; j=n
\end{cases}
\end{equation}
\end{proof}
Using the orthogonality relations in \eqref{ortho1} and  \eqref{ortho2}, we can easily prove the following inverse relation.
\begin{thm}
The following inverse relations hold:
\begin{align}
f_n&=\sum_{k=0}^nw_{m,r}[n,k]_qg_k\Leftrightarrow g_n=\sum_{k=0}^nW_{m,r}[n,k]_qf_k \label{inv_rel}\\
f_k&=\sum_{n=0}^{\infty}w_{m,r}[n,k]_qg_n\Leftrightarrow g_k=\sum_{n=0}^{\infty}W_{m,r}[n,k]_qf_k.\label{inv_rel1}
\end{align}
\end{thm}

\smallskip
\begin{rema}\rm
Using the inverse relation in \eqref{inv_rel1} and the exponential generating function in \eqref{exp_qrWhit}, we can easily obtain the following identity
\begin{equation}\label{iden1}
	\sum_{n\geq 0}\frac{[k]_q!w_{m,r}[n,k]_q[\Delta_{q^m,m^{n}}e_q([x+jm+r]_q[t]_q)]_{x=0}}{[t]^k_q[n]_{q^m}![m]^n_q} =1.
\end{equation}
Also, using the inverse relation in \eqref{inv_rel1} and the rational generating function in \eqref{rationalGF}, we obtain
\begin{equation}\label{iden2}
\sum_{n\geq k}\frac{w_{m,r}[n,k]_qq^{m\binom{n}{2}+nr}[t]^{n-k}_q}{\prod_{j=0}^{n}(1-[mj+r]_q[t]_q)}
	=1.
\end{equation}
Thus, combining \eqref{iden1} and \eqref{iden2} yields
\begin{equation}
	\sum_{n\geq 0}w_{m,r}[n,k]_q\left\{\frac{[k]_q![\Delta_{q^m,m^{n}}e_q([x+jm+r]_q[t]_q)]_{x=0}}{[n]_{q^m}![m]^n_q}-\frac{q^{m\binom{n}{2}+nr}[t]^{n}_q}{\prod_{j=0}^{n}(1-[mj+r]_q[t]_q)}\right\} =0.
\end{equation}
\end{rema}

Now, we can establish the relation between $(q,r)$-Whitney-Lah numbers and both kinds of $(q,r)$-Whitney numbers.
\begin{thm}
The $(q,r)$-Whitney-Lah numbers satisfy the following relation
\begin{equation}\label{LahSN}
L_{m,r}[n,j]_q=\sum_{k=j}^nw_{m,-r}[n,k]_qW_{m,r}[k,j]_q.
\end{equation}
\end{thm}
\begin{proof}
Using the horizontal generating functions for $(q,r)$-Whitney numbers and $(q,r)$-Whitney numbers, we have
\begin{align*}
\sum_{k=0}^nL_{m,r}[n,k]_q[t-r|m]_{k,q}&=[t+r|m]_{n,q}=\sum_{k=0}^nw_{m,-r}[n,k]_q[t]_q^k\\
&=\sum_{k=0}^nw_{m,-r}[n,k]_q\sum_{j=0}^kW_{m,r}[k,j]_q[t-r|m]_{j,q}\\
&=\sum_{j=0}^n\sum_{k=j}^nw_{m,-r}[n,k]_qW_{m,r}[k,j]_q[t-r|m]_{j,q}\\
\sum_{j=0}^nL_{m,r}[n,j]_q[t|m]_{j,q}&=\sum_{j=0}^n\left(\sum_{k=j}^nw_{m,-r}[n,k]_qW_{m,r}[k,j]_q\right)[t-r|m]_{j,q}.
\end{align*}
Comparing the coefficients of $[t|m]_{j,q}$ completes the proof.
\end{proof}
The following theorem contains the main result of this paper.

\smallskip
\begin{thm}
The explicit formula for the first form $(q,r)$-Dowling numbers $D_{m,r}[n]_q$ is given by
$$D_{m,r}[n]_q=\sum_{k=0}^nW_{m,-r}[n,k]_q\left(\sum_{j=0}^kL_{m,r}[k,j]_q\right),$$
\end{thm}
\begin{proof}
Using the inverse relation \eqref{inv_rel} with 
$$f_n=L_{m,r}[n,j]_q, \;\;\;\;g_k=W_{m,r}[k,j]_q,$$
relation \eqref{LahSN} can be transformed as
\begin{equation}\label{LahSN-1}
W_{m,r}[n,j]_q=\sum_{k=j}^nW_{m,-r}[n,k]_qL_{m,r}[k,j]_q.
\end{equation}
Then summing up both sides over $j$ yields
\begin{align*}
D_{m,r}[n]_q&=\sum_{j=0}^nW_{m,r}[n,j]_q=\sum_{j=0}^n\sum_{k=j}^nW_{m,-r}[n,k]_qL_{m,r}[k,j]_q\\
&=\sum_{k=0}^nW_{m,-r}[n,k]_q\left(\sum_{j=0}^kL_{m,r}[k,j]_q\right),
\end{align*}
which is exactly the desired explicit formula for $(q,r)$-Dowling numbers.
\end{proof}
The above explicit formula may also be called a Qi-type formula for $(q,r)$-Dowling numbers, which is analogous to explicit formula obtained F. Qi \cite{Qi}.

\section{Conclusion}
The explicit formula for the first form $(q,r)$-Dowling numbers has already been derived by establishing an appropriate $(q,r)$-Whitney-Lah numbers and $(q,r)$-Whitney numbers of the first kind. To establish explicit formulas for the two other forms of $(q,r)$-Dowling numbers, it entails defining appropriate versions of $(q,r)$-Whitney-Lah numbers and $(q,r)$-Whitney numbers of the first kind. It is left to the readers to derive these explicit formulas.

\bigskip
\noindent{\bf Acknowledgement}. This research has been funded by Cebu Normal University (CNU) and the Commission  on Higher Education - Grants-in-Aid for Research (CHED-GIA).

\end{document}